\documentclass[12pt,oneside,reqno]{amsart}
\usepackage[all]{xy}
\usepackage{amsfonts,amsmath,oldgerm,amssymb,amscd}
\UseComputerModernTips
\numberwithin{equation}{section}
\usepackage[breaklinks]{hyperref}
\allowdisplaybreaks

\usepackage{enumerate}

\usepackage{caption} 
\newtheorem{theorem}{Theorem}[section]
\newtheorem{proposition}[subsection]{\bf Proposition}

\newtheorem{lemma}[subsection]{{\bf Lemma}}

\newcommand{\al}{\alpha}
\newcommand{\be}{\beta}

\begin{document}

\title[On perfect powers that are sum of two Balancing Numbers]{On perfect powers that are sum of two Balancing Numbers} 

\author[P. K. Bhoi]{P. K.  Bhoi}
\address{Pritam Kumar Bhoi, Department of Mathematics, National Institute of Technology Rourkela-769008, India.}
\email{pritam.bhoi@gmail.com}

\author[S. S. Rout]{S. S. Rout}
\address{Sudhansu Sekhar Rout, Department of Mathematics, National Institute of Technology Calicut- 673 601, 
Kozhikode, India.}
\email{sudhansu@nitc.ac.in; lbs.sudhansu@gmail.com}

\author[G. K. Panda]{G. K. Panda}
\address{Gopal Krishna Panda, Department of Mathematics, National Institute of Technology Rourkela-769008, India.}
\email{gkpanda@nitrkl.ac.in}

\maketitle
\pagenumbering{arabic}
\pagestyle{headings}

\begin{abstract}
Let $B_k$ denote the $k^{th}$ term of balancing sequence. In this paper we find all positive integer solutions of the Diophantine equation $B_n+B_m = x^q$ in variables $(m, n, x, q)$ under the assumption $n\equiv m \pmod 2$. Furthermore, we study the Diophantine equation 
\[B_n^{3}\pm B_m^{3} = x^q\] with positive integer $q\geq 3$ and $\gcd(B_n, B_m) =1$.
\end{abstract}

\section{Introduction}

A balancing number $B$ is a natural number which satisfies the Diophantine equation
\begin{equation}\label{eq1}
1 + 2+ \dots +(B - 1) = (B + 1)+ \dots +(B + R).
\end{equation}
where $R$ is a natural number. Here $R$ is called balancer corresponding to $B$ (see \cite{bp}). If $B$ is a balancing number, then $8B^2 + 1$ is a perfect square and its positive square root is called a Lucas-balancing number (see \cite{gkp} and \cite{r15}). The $n^{th}$ balancing and Lucas-balancing numbers are denoted by $B_n$ and $C_n$ respectively. 
The balancing sequence $(B_n)_{n \geq 0}$ is a binary recurrence sequence with initial values $B_0$ = $0$,  $B_1$ = $1$ and satisfies the recurrence relation
\begin{equation}\label{eq2}
B_{n} = 6B_{n-1} - B_{n-2} \quad \mbox{for all} \quad n \geq 2.
\end{equation}
The Lucas-balancing sequence $(C_n)_{n \geq 0}$ is a binary recurrence sequence with initial values $C_0$ = $1$,  $C_1$ = $3$ and satisfies the same recurrence relation
\begin{equation}\label{eq11}
C_{n} = 6C_{n-1} - C_{n-2} \quad \mbox{for all} \quad n \geq 2.
\end{equation}
The Binet formulas for balancing number and Lucas-balancing number are given by
\begin{equation}\label{eqbinet}
B_n = \frac{\al^{n}-\be^{n}}{4\sqrt{2}},~C_n = \frac{\al^{n}+\be^{n}}{2},~ \mbox{for} ~ n~=~0,~1,~2~ \dots
\end{equation}
where $\al = 3+ 2\sqrt{2}$ and $\be = 3- 2\sqrt{2}$. For more information about balancing numbers and its generalization, one may refer to \cite{r15}.

There is a long history of Diophantine equations involving perfect powers and  binary recurrence sequence. Finding perfect powers in binary recurrence sequence is  very interesting. Recently, Bugeaud et al. \cite{bm} proved that 0, 1, 8, and 144 are the only perfect powers in the Fibonacci sequence using linear forms in logarithm and modular approach. Similarly, perfect powers in balancing and Lucas balancing sequence have been studied (see \cite{dr}). Recently, the Diophantine equation
\begin{equation}\label{12}
F_n \pm F_m = y^q,
\end{equation}
where $F_n$ is $n^{th}$ Fibonacci number, $n\geq m \geq 0, y\geq 2$ and $q\geq 2$ has been studied by a number of authors. Luca and Patel \cite{lp} proved that if $n\equiv m \pmod 2$, then either $n\leq 36$ or $y =0$ and $n=m$. This problem is still open for $n \not \equiv m \pmod 2$. Kebli et al. \cite{kkll} proved that there are only finitely many integer solutions $(n, m, y, q)$ with $y, q\geq  2$ of \eqref{12} using {\em abc conjecture}. Further, in \cite{zt} Zhang and Togb\'e studied the Diophantine equations 
\begin{equation}
F_n^q\pm F_m^q = y^p
\end{equation}
with positive integers $q, p\geq 2$ and $\gcd(F_n, F_m) =1$. Also, perfect powers that are sums of two Pell numbers have been studied (see \cite{ahrt}). Recently, in \cite{brp} Bhoi et al., study the Diophantine equation $U_n + U_m = x^q$ in integers $n \geq m \geq 0$, $x \geq 2$, and $q \geq 2$, where $(U_k)_{k\geq 0}$ is Lucas sequence of first kind. In particular, they proved that  there are only finitely many of them for a fixed $x$ using linear forms in logarithms and that there are only finitely many solutions in $(n, m, x, q)$ with $q, x\geq 2$ under the assumption of the {\em abc conjecture}.

In this paper, we prove the following results: 
\begin{theorem}\label{thm1}
 The only positive integer solution of the Diophantine equation 
 \begin{equation}\label{eq5}
B_n+B_m = x^q, ~q \geq 2
\end{equation} 
in $(n, m, x, q)$ with $n\equiv m \pmod {2}$ is 
 $(n, m, x, q)$ = $(3,1,6,2)$, that is, 
 \[B_3+ B_1= 35+1= 6^2.\]
\end{theorem}

\begin{theorem}\label{thm2}
	The solutions of the Diophantine equation 
		\begin{equation}\label{eq12}
	B_n^2-B_m^2 = x^q \quad\mbox{with}\quad \gcd(B_n, B_m) =1, \quad q \geq 2
	\end{equation}
	in integers $(n, m, x, q)$  with $n> m \geq 0 $ and $x>0$ are
	$(n, m, x, q)$ = $(1, 0, 1, k)$, with $k \geq 2$ and $(2, 0, 6, 2)$. 
	
\end{theorem}

\begin{theorem}\label{thm3}
	
	The only solution of the Diophantine equation 
	\begin{equation}\label{eq13}
	B_n^3 \pm B_m^3 = x^q \quad\mbox{with}\quad \gcd(B_n, B_m) =1, \quad q \geq 3
	\end{equation}
	in integers $(n, m, x, q)$ with $n > m \geq 0 $ and $x>0$ is
	$(n, m, x, q)$ = $(1, 0, 1, k)$, with $k \geq 3$.
	
\end{theorem}
We organise this paper as follows. In Section \ref{sec2}, we recall and prove some results that will be useful for the proofs of main theorems. In Section \ref{sec3}, we will prove Theorem \ref{thm1}-\ref{thm3}. Finally, we finish this paper with a concluding remark.

 \section{Auxiliary results}\label{sec2}

\begin{lemma}\label{lem1}
Assume that $n \equiv m \pmod 2$. Then
\[B_n+B_m = 2 B_{(n+m)/2}C_{(n-m)/2}.\]
Similarly,
\[B_n-B_m = 2 B_{(n-m)/2}C_{(n+m)/2}.\] 
\end{lemma}

\begin{proof}
By \cite[Theorem 2.5]{gkp}, we know that if $x$ and $y$ are natural numbers, then 
\begin{equation*}
B_{x+y} = B_xC_y + C_xB_y
\end{equation*} and for $x>y$
\begin{equation*}
B_{x-y} = B_xC_y - C_xB_y.
\end{equation*}	
 Setting $x+y = n$ and $x-y = m$ in the above equations and since $n \equiv m \pmod 2$, we get
 \[ B_n + B_m = 2 B_{(n+m)/2}C_{(n-m)/2}.\]
 and  \[B_n-B_m = 2 B_{(n-m)/2}C_{(n+m)/2}.\] 
	\end{proof}
Before proceeding further, we define two more binary recurrence sequences which are related to balancing sequence. The Pell sequence $(P_n)_{n\geq 0}$  is defined recursively as 
\[P_{n+1} = 2P_n + P_{n-1}, \quad \mbox{for}\;\; n= 1, 2, \ldots\] with initial values $P_0=0,\; P_1=1$ and  the associated Pell sequence $(Q_n)_{n\geq 0}$ is defined as 
\[Q_{n+1} = 2Q_n + Q_{n-1}, \quad \mbox{for}\;\; n= 1, 2, \ldots\] with initial values $Q_0=1,\; Q_1=1$. \begin{lemma}[Theorem 3.1, \cite{pkr11}]\label{lemprod}
For $n =0, 1, \ldots$
\begin{equation}\label{eqprod1}
B_m=P_m Q_m
\end{equation}
where $P_m$ and $Q_m$ are the $m$-th Pell and the $m$-th associated Pell numbers, respectively.
\end{lemma}

Note that except $B_1 =1$, there are no other perfect powers in the sequence of balancing numbers.
\begin{lemma}[Prop. 3.1, \cite{dr}]\label{lembal}
For any positive integers $y$ and $l\geq 2$,  the equation 
\begin{equation}\label{eq3}
B_m=y^l
\end{equation}
has no solution for integers $m\geq 2$.
\end{lemma}

\begin{lemma}[Prop. 3.2, \cite{dr}]\label{lemlucbal}
For any positive integers $y$ and $l$ with $l\geq 2$,  the equation 
\begin{equation}\label{eq4}
C_n=y^l
\end{equation}
has no solutions for integers $n\geq 1$.
\end{lemma}

\begin{lemma}\label{lem2}
If
\begin{equation}\label{eq6}
 B_n = 2^sx^b. 
\end{equation}
for some integers $n \geq 1$, $x \geq 1$, $ b \geq 2$ and $ s \geq 0$, then $ n = 1$.	 
\end{lemma}
\begin{proof}
By Lemma \ref{lemprod}, we have $B_n = P_nQ_n$ and note that $\gcd(P_n ,Q_n)=1$ (see \cite[Chapter 7]{koshy}). Let $x=x_1x_2$ with $\gcd$$(x_1,x_2)$=1. Then from \eqref{eq6}, we get \[P_nQ_n = 2^sx_1^bx_2^b.\] So, we have the following cases: $P_n = x_1^b, \; Q_n = 2^sx_2^b$ and $P_n = 2^sx_1^b,\; Q_n = x_2^b$. If $P_n = x_1^b$ and $Q_n = 2^sx_2^b$,  then by  \cite[Lemma 2.6]{bdgl}  $n= 1$. In the later case, by  \cite[Lemma 2.6]{bdgl} we have $n \in \{1,2,7\}$ and among these values of $n$, only $n =1$ satisfies $Q_n =x_2^b$. This completes the proof of lemma.
\end{proof}

\begin{lemma}\label{lem3}
	If
	\begin{equation}\label{eq7}
	C_n = 2^sx^b. 
	\end{equation}
for some integers 	$ n \geq 1$, $x \geq 1$, $ b \geq 2$ and $ s \geq 0$, then no solution exists.	 
\end{lemma}
\begin{proof}Recall that the Lucas-balancing sequence $(C_n)_{n \geq 0}$ with initial values $C_0$ = $1$,  $C_1$ = $3$, satisfies the recurrence relation \eqref{eq11}. First we claim that all the Lucas balancing numbers are odd. Suppose on contrary $t\geq 2$ is the smallest index such that $C_t$ is even. Then from \eqref{eq11}, we get $C_{t-2} = 6C_{t-1}-C_t$ is even, which is a contradiction. Thus, all the Lucas balancing numbers are odd integers and hence there does not exists any solution of \eqref{eq7}.
\end{proof}
The following result can be found in \cite{Mcdaniel}.
\begin{lemma}\label{lem4}
Let $n=2^an_1$ and $m=2^bm_1$ be two positive integers with $n_1$ and $m_1$ odd integers and $a$ and $b$ non-negative integers. Let $d=\gcd(n,m)$. Then 
	\begin{enumerate}
		\item $\gcd (B_n,B_m)$ = $B_d,$
		\item $\gcd (C_n,C_m)$ = $C_d$ if $a=b$ and is 1  otherwise,
		\item $\gcd (B_n,C_m)$ = $C_d$ if $a>b$ and is 1  otherwise.
	\end{enumerate}
\end{lemma}

\begin{lemma}\label{lem5}
Let p be a prime. If (a,b,c) is an integer solution of the equation 
	\[x^3+y^3=z^p, \quad p \geq 3\]
with $\gcd(a, b) =1, abc \neq 0$ and $2|ac$. Then $3|c $ and $2|a$ but $4  \nmid a$.	 
\end{lemma}
\begin{proof}
For $p=3$, it is a classical result. When $p\geq 17$ is a prime, see \cite{kraus}. When $p= 5, 7, 11, 13$, it can be obtained from the result of Bruin \cite{bruin} and Dahmen \cite{dahmen}.
\end{proof}

The following result is an easy exercise in elementary number theory.
\begin{lemma}\label{lem6a}
	Let p be an odd prime, $x, y, z, k$ integers  with $\gcd(x, y) =1$. If
\[x^p+y^p=z^k, \quad k \geq 2,\]
then $x+y=c^k $ or $p^{k-1}c^k$ for some integer c. 
\end{lemma}

\begin{lemma}\label{lem7}
	If
\begin{equation}\label{eq15}
	B_n = 3^sx^b. 
\end{equation}
for some integers $n\geq 1$, $x \geq 1$, $ b \geq 2$ and $ s \geq 0$, then $ n = 1$.	 
\end{lemma}
\begin{proof}
Let $x=x_1x_2$ with $\gcd$$(x_1,x_2)$=1. Then from the relation $B_n = P_n Q_n$ and \eqref{eq15}, we get $P_nQ_n = 3^sx_1^bx_2^b$. We have two cases: $P_n = x_1^b$ and $Q_n= 3^s x_2^b$ or $P_n = 3^sx_1^b$ and $Q_n=  x_2^b$. So from $P_n = x_1^b$, we have $n= 1$ or $7$ and then substituting the values of $n$ in $Q_n= 3^s x_2^b$, we get $n= 1, s=0, x_2 = 1, b=0$. So, altogether $n =1$. In the case,  $P_n = 3^sx_1^b$ and $Q_n= x_2^b$ we also have $n=1$.
\end{proof}

\begin{lemma}[Prop. 3.3, \cite{dr}]\label{lemlu3}
For any positive integers $y, k$ and $l$ with $l\geq 2$,  the equation 
\begin{equation}\label{eq4}
C_n=3^ky^l
\end{equation}
has no solutions for integers $n\geq 2$.
\end{lemma}

We call a natural number $t$ the period of the balancing sequence modulo $\mu$ if $B_t\equiv 0, B_{t+1} \equiv 1\pmod{\mu}$ and for if for some natural number $k, \; B_k\equiv 0, B_{k+1} \equiv 1\pmod{\mu}$, then $t$ divides $k$ (see \cite{BPP, PR}).
\begin{lemma}\label{lemper}
The balancing sequence have the following divisibility properties (see \cite[Theorem 5.1]{PR}):
\begin{align*}
& 2\mid B_n \iff n\equiv 0\pmod{2};\\
& 4\mid B_{n}\iff n\equiv 0\pmod{4}.
\end{align*}
Further, the residue of $B_n$ modulo $9$ depends on the residue of $n$ modulo $12$ as follows:
\begin{align*}
&B_n \equiv 0\pmod{9} \iff n\equiv 0, 6 \pmod{12},\\
&B_n \equiv 1\pmod{9} \iff n\equiv 1, 5, 9 \pmod{12},\\
&B_n \equiv 3\pmod{9} \iff n\equiv 8, 10 \pmod{12},\\
&B_n \equiv 6\pmod{9} \iff n\equiv 2, 4 \pmod{12},\\
&B_n \equiv 8\pmod{9} \iff n\equiv 3, 7, 11 \pmod{12}.
\end{align*}
\end{lemma}
\begin{lemma}[Theorem 5.1, \cite{PR}]\label{lem2.14}
For any natural number $2^k \mid n$ if and only if $2^k\mid B_n$.
\end{lemma}

\begin{proposition}\label{prop1}
 The only positive integer solution of the Diophantine equation	\begin{equation}\label{eq8}
  B_NC_M = 2^px^q  
 \end{equation}
 with $N$, $M$, $x$ positive integers, $p \geq 0$, $q \geq 2$ is $(N, M)$ = $(2,1)$. 
\end{proposition}

\begin{proof}
Put $N=2^gN_a$ and  $M=2^hM_a$, where $N_a$, $M_a$ are odd and $g$ and $h$ are non-negative integers. By Lemma \ref{lem2.14}, $2^g\mid B_n$, that is, $B_n = 2^g k_1$ for some integer $k_1$. If $g \leq h$, then by Lemma \ref{lem4}, we know that $\gcd (B_N, C_M)$ = 1. Hence, $C_M$ = $x_2^q$ with $2\nmid x_2$, which has no solution. So, in this case, solution does not exists.

 Hence, we may assume that $g>h$. Let $g-h>0$ and suppose $d = \gcd(N, M)$. Therefore $d = 2^h\gcd(N_a,M_a)$. Write $N=2^tdr$, where $r$ is an odd integer. Then by Lemma \ref{lem4} and using $B_{2n} = 2B_n C_n$, we get  
 \begin{align*}
 2^px^q & = B_NC_M \\
 &= B_{2^tdr}C_M \\
 &= B_{2\cdot 2^{t-1}dr}C_M \\
 &= 2 B_{2^{t-1}dr}C_{2^{t-1}dr}C_M \\
 & = 2^2 B_{2^{t-2}dr}C_{2^{t-2}dr}C_{2^{t-1}dr}C_M \\
 &=\cdots \\
 &= 2^tB_{dr}\cdot C_{dr}\cdot C_{2dr} \dots C_{2^{t-1}dr}\cdot C_M.
 \end{align*}
 Note that $\upsilon_2{(dr)} = \upsilon_2{(M)}$ and $\upsilon_2{(dr)} \leq \upsilon_2{(2^idr)}$ for $i \geq 0$.
Thus by lemma \ref{lem4}(3), we get 
\[\gcd (B_{dr}, C_{dr}\cdot C_{2dr} \dots C_{2^{t-1}dr}\cdot C_M) =1.\]
 So, 
 \[B_{dr} = x_1^q \quad \mbox{or} \quad 2^ux_1^q, \quad C_{dr}\cdot C_{2dr} \dots C_{2^{t-1}dr}\cdot C_M = x_2^q \quad \mbox{and}\;\; x_1x_2 = x.\]
If $B_{dr} = 2^u x_1^q$ with $u\geq 0$,  then by Lemma \ref{lem2}, we get $dr = 1$.
Thus from
 \[2^px^q = 2^tB_{dr}\cdot C_{dr}\cdot C_{2dr} \dots C_{2^{t-1}dr}\cdot C_M.\]
 We get
\[2^px^q = 2^t \cdot C_{1}\cdot C_{2} \dots C_{2^{t-1}}\cdot C_M.\]
 Here, in the right hand side all terms are odd except $2^t$. Hence, $p=t$. Now let $t \geq 2$. Then $\upsilon_2{(M)} < \upsilon_2{(2^{t-1}dr)}$. Using Lemma \ref{lem4}, we get 
\[\gcd (C_{2^{t-1}dr}, C_{dr}\cdot C_{2dr} \dots C_{2^{t-2}dr}\cdot C_M) =1.\]
Thus, 
\[C_{2^{t-1}dr} = x_3^q, C_{dr}\cdot C_{2dr} \dots C_{2^{t-2}dr}\cdot C_M= x_4^q,\quad \mbox{and}\;\; x_3x_4 = x_2.\] 
Then $2^{t-1}dr =0$, which is impossible.
Thus, $t = 1$. Hence, $N = 2$, and so $B_N=6$. Now $B_NC_M$ = $2^px^q$.
Here putting the value of $B_N$, we get $6C_M = 2^px^q$. which gives $3C_M = 2^{p-1}x^q$. So $C_M = 2^{p-1}3^{q-1}x_3^q,$ as $3 | x$. Thus if $p>1$, then solution does not exist.  If $p=1$, then $C_M=3^{q-1}x_3^q$. So, by Lemma \ref{lemlu3}, we get $M=1$. Hence, $(M,N)$ = $(1,2).$ If $p=0$, then $t =0$, and hence $N =dr = 1$, and so $B_N=1$. This implies $C_M = x^q$, which has no solution. This completes the proof.
\end{proof}

Now, we will give the proof of our main result.
\section{Proof of main theorems}\label{sec3}
\subsection{Proof of Theorem \ref{thm1}}
 If either $n=0$ or $m=0$, then the theorem follows from Lemma \ref{lembal}. If $n=m$, then the \eqref{eq5} becomes $2B_n = x^q$, which can also be written as $B_n = 2^{q-1}x_1^q$. From Lemma \ref{lem2}, we get $n=1$. Thus, we may assume that $n>m>0$. Since $n\equiv m \pmod{2}$, then by Lemma \ref{lem1}, we get
\begin{equation}\label{eq132}
x^q=B_n+B_m = 2B_NC_M,
\end{equation}
where $N= \frac{n+m}{2}$ and $M= \frac{n-m}{2}$ (here $N$ and $M$ both are positive). So from \eqref{eq132}, $2\mid x$, that is $x = 2x_1$ for some integer $x_1$. Thus, \eqref{eq132} becomes
\begin{equation}\label{eq133}
2^{q-1}x_1^q=B_NC_M,
\end{equation}
Using Proposition \ref{prop1}, we get $N=2$ and $M=1$ and this implies $n =3$ and $m=1$. This completes the proof. \qed

\subsection{Proof of Theorem \ref{thm2}}
For any non-negative integers  $n$ and $m$, we have
\[x^q=B_n^2 - B_m^2 =B_{n+m}B_{n-m}.\]
Since $\gcd(B_n, B_m) =1$, we get $\gcd(n, m) =1$. This implies $\gcd(n+m, n-m) =1$ or $2$. Suppose $\gcd(n+m, n-m) =1$. By Lemma \ref{lem4} , \[\gcd(B_{n+m},B_{n-m}) =B_{\gcd(n+m,n-m)} = B_1=1.\] Thus we have, 
\[B_{n+m}=u^q,\quad  B_{n-m}=v^q, \quad\mbox{and}\quad x = uv.\]
By Lemma \ref{lembal}, we get $n+m = 1$ and $n-m=1$ and hence $(n, m, x, q) = (1, 0, 1, q)$. Next consider the case, $\gcd(n+m, n-m) =2$. In this case, 
\[\gcd(B_{n+m},B_{n-m}) = B_2 = 6.\] So,
\[B_{n+m} = 6x_1^q, B_{n-m} = 6^{q-1}x_2^q; \quad \mbox{or} \quad B_{n+m} = 6^{q-1}x_1^q, B_{n-m} = 6x_2^q.\]  If  $B_{n+m} = 6x_1^q$ and $B_{n-m} = 6^{q-1}x_2^q$, then from Lemma \ref{lem7} and Lemma \ref{lemlu3}, we get $n+m= 2$ and $n-m =2$. In this case, we get $(n, m, x, q) = (2, 0, 6, 2)$. This completes the proof. \qed

\subsection{Proof of Theorem \ref{thm3}}
First assume the case $n \equiv m \pmod{2}$ with $n>m$. Since $\gcd(B_n,B_m)=1$, then from Lemma \ref{lem6a}, we get the following two cases:
\begin{enumerate}
	\item $B_n \pm B_m = x^q$;
	\item  $B_n \pm B_m = 3^{q-1}x^q$.
\end{enumerate} 
For the first case, we know the solution is $(n, m, x, q)$ = $(3,1,6,2)$. However, $B_3^3 \pm B_1^3 \neq x^q$.
Hence, there is no solution for this case.

Now consider  $B_n \pm B_m = 3^{q-1}x^q$.
Since $n\equiv m \pmod{2}$, by Lemma \ref{lem1}, $B_n \pm B_m = 2B_NC_M$, where
\[N= \frac{n\pm m}{2} \quad \mbox{and} \quad M= \frac{n \mp m}{2}.\] So, $2\mid x$, that is, $x=2y$ for some integer $y$. Hence,
\[B_NC_M = 2^{q-1}\cdot 3^{q-1} y^q.\]
As $\gcd(B_n,B_m)=1$, thus $\gcd(n,m)=1$, so we have $\gcd(N,M)=1$. Thus, by Lemma \ref{lem4}, $\gcd(B_N,C_M)$ = $3$ or $1$.

First, we consider  $\gcd(B_N, C_M)$ = $3$. Since $C_k$ is odd for any $k\geq 0$,  we have 
\[B_N = 2^{q-1}\cdot 3\cdot x_1^q \quad \mbox{and}\quad C_M = 3^{q-2}\cdot x_2^q \quad\mbox{with} \quad 3\nmid x_1x_2, \; x_1x_2 =y.\] or 
\[B_N =  2^{q-1}\cdot 3^{q-2}\cdot x_1^q \quad \mbox{and}\quad C_M = 3\cdot x_2^q \quad\mbox{with} \quad 3\nmid x_1x_2, \; x_1x_2 =y.\] 
Thus from Lemma \ref{lem7} and Lemma \ref{lemlu3}, we get $N=2 , q=2$ and  $M=1 , q=3$. So, 
there is no solution of $B_n \pm B_m = 3^{q-1}x^q$.

Next consider, $\gcd(B_N, C_M)$ = $1$,
then we have 
\begin{itemize}
\item $B_N = 2^{q-1}y_1^q$ and $C_M = 3^{q-1}y_2^q$ with $2\nmid y_2, 3\nmid y_1$ and $y_1y_2=y, \gcd(y_1, y_2) =1$.
\item $B_N = 3^{q-1}y_1^q$ and $C_M = 2^{q-1}y_2^q$ with $2\nmid y_1, 3\nmid y_2$ and $y_1y_2=y, \gcd(y_1, y_2) =1$.
\end{itemize}
In the first case $N=1, q=1$ and $M=1, q=2$ and second case is not possible as Lucas balancing numbers are always odd.Thus, there does not exist any solution of \ref{eq13}. 
 
 Now assume that $n \not \equiv m \pmod{2}$ with $n>m$. If $m =0$, then $n=1$ since $\gcd(B_n, B_m) =1$. So the solution is $(n, m, x, q)=(1, 0, 1, k)$, where $k \geq 3$. Thus, we assume $m\geq 1$, which gives $xB_nB_m \neq 0$ and $\gcd(B_n, B_m) =1$. By Lemma \ref{lem5}, we have $3\mid x$ and by Lemma \ref{lem6a}, $B_n \pm B_m = 3^{q-1}z^q$. As $q\geq 3$, we deduce that 
 \begin{equation}\label{eqthm1.3}
 9\mid (B_n \pm B_m), \quad \mbox{with}\;\; 2\mid B_n,\; 4\nmid B_n.
 \end{equation}
 Further, by Lemma \ref{lemper}, $9\mid (B_n+B_m)$ if and only if
 \begin{enumerate}
  \item $n\equiv 0, 6\pmod{12}$ and $m\equiv 0, 6\pmod{12}$.
  \item $n\equiv 1, 5, 9\pmod{12}$ and $m\equiv 3, 7, 11\pmod{12}$.
  \item $n\equiv 8, 10\pmod{12}$ and $m\equiv 2, 4\pmod{12}$.
 \end{enumerate}
 Since $n \not \equiv m \pmod{2}$, the above cases (1) , (2) and (3) will not hold. Again, $9\mid (B_n -B_m)$ if and only if $n\equiv 0, 6\pmod{12}$ and $m\equiv 0, 6\pmod{12}$ and this not true. Thus, $B_n\pm B_m\not \equiv 0\pmod{9}$, which contradicts \eqref{eqthm1.3}. This completes the proof of Theorem \ref{thm3}. \qed
 
\section{Concluding Remark}
For $n\equiv m \pmod 2$ we find all solutions to \eqref{eq5}. Finding all solutions  to \eqref{eq5} when $n\not \equiv m \pmod{2}$ is still an open problem. Note that under the assumption, $n\not \equiv m\pmod{2}$, no factorization is known for the left hand side of  \eqref{eq5}. Further, to solve a more general Diophantine equation of the form
\[B_n^p+B_m^p= x^q\] in integers $(n, m, x, p, q)$, one need to know integral solutions of equations of the shape
\begin{equation}\label{eqcr}
B_n = p^a z^q,\;\; \mbox{and}\;\;C_n = p^a z^q
\end{equation} with $p$ prime, $q\geq 2, a>0$.  It is interesting to find all explicit solutions (if any) to \eqref{eqcr}. 

{\bf Acknowledgment:} The authors sincerely thank the referee for his/her thorough reviews and very helpful comments and suggestions which significantly improves the paper. The first author's work is supported by CSIR fellowship(File no: 09/983(0036)/2019-EMR-I).

\end{document}